\documentclass[12pt]{amsart}
\usepackage{mathrsfs}

\font\bbbld=msbm10 scaled\magstephalf

\usepackage{hyperref}

\usepackage{hyperref}
\hypersetup{colorlinks, urlcolor = blue}

\usepackage{graphicx}

\usepackage{xcolor}

\newcommand{\bi}{\bar{i}}

\newcommand{\bpartial}{\bar{\partial}}

\def \a{\alpha}

\def \p{\partial}
\def \f{\frac}

\newcommand{\fRe}{\mathfrak{Re}}

\newcommand{\bfR}{\hbox{\bbbld R}}

\newcommand{\tr}{\mbox{tr}}

\newcommand{\ol}{\overline}
\newcommand{\ul}{\underline}

\newtheorem{theorem}{Theorem}[section]
\newtheorem{lemma}[theorem]{Lemma}
\newtheorem{proposition}[theorem]{Proposition}

 \theoremstyle{definition}

\theoremstyle{remark}

\numberwithin{equation}{section}

\begin{document}
\setlength{\baselineskip}{1.2\baselineskip}

\title[Parabolic Complex Monge-Amp\`ere Type Equation]
{Parabolic Complex Monge-Amp\`ere Type Equations on Closed Hermitian Manifolds}

\author{Wei Sun}

\address{Department of Mathematics, Ohio State University,
         Columbus, OH 43210}
\email{sun@math.ohio-state.edu}

\begin{abstract}
We study the parabolic complex Monge-Amp\`ere type equations on closed Hermitian manfolds. We derive uniform $C^\infty$ {\em a priori} estimates for normalized solutions, and then prove the $C^\infty$ convergence. The result also yields a way to carry out method of continuity for elliptic Monge-Amp\'ere type equations.

\end{abstract}

\maketitle

\bigskip

\section{Introduction}
\label{pmach-int}

Let $(M^n, \omega)$ be a compact Hermitian manifold of complex dimension $n\geq 2$ and $\chi$ a smooth real $(1,1)$ form on $M^n$. Write $\omega$ and $\chi$ respectively as
\[
\omega = \frac{\sqrt{-1}}{2} \sum_{i,j} g_{i\bar j} dz^i \wedge d\bar z^j,
\]
and
\[
\chi = \frac{\sqrt{-1}}{2} \sum_{i,j} \chi_{i\bar j} dz^i \wedge d\bar z^j.
\]
We denote $\chi_u : = \chi + \frac{\sqrt{-1}}{2} \p\bpartial u$, and also set 
\[ 
[\chi] := \big\{\chi_u:  \, u \in C^2 (M)\big\}, \quad [\chi]^+ := \big\{\chi' \in [\chi]: \chi' > 0\}. 
\]

In this paper, we are concerned with the following equation on $M$,
\begin{equation}
\label{pmach-parabolic-equation}
	\frac{\p u}{\p t} = \ln  \frac{\chi_u^n }{\chi_u^{n - \a} \wedge \omega^\a} - \ln \psi 
\end{equation}
with initial value $u(x , 0) = 0$. We are interested in solving the corresponding nondegenerate parabolic equations.  To be parabolic,  equation~\eqref{pmach-parabolic-equation} must have $\chi_{u} > 0$; we call such functions {\em admissible} or $\chi$-{\em plurisubharmonic}. To be nondegenerate, we need to assume $\psi > 0$ on $M$.

In order to solve equation~{pmach-parabolic-equation}, it is necessary to impose some condition. As in \cite{GSun12} and \cite{Sun2013e}, we define,  for a smooth positive real function $\psi$ on $M$,
\begin{equation}
\label{pmach-definition-cone-condition}
\mathscr{C}_\a (\psi) := \{[\chi]: \exists \chi' \in [\chi]^+, n\chi'^{n-1} > (n - \a)\psi \chi'^{n - \a - 1} \wedge \omega^\a\}.
\end{equation}
If $\chi \in \mathscr{C}_\a (\psi)$, we say that $\chi$ satisfies the cone condition.

The study of the parabolic flows is motivated by the complex Monge-Am\`ere type equation
\begin{equation}
\label{pmach-elliptic-equation}
 \chi_u^n =  \psi \chi_u^{n-\alpha} \wedge \omega^{\alpha}, \; \chi_u > 0 .
\end{equation}
A critical point of the flow gives a Hermitian metric $\chi$ satisfying
\begin{equation}
\label{pmach-elliptic-equation-actual}
 \chi_u^n =  e^b \psi \chi_u^{n-\alpha} \wedge \omega^{\alpha}, \; \chi_u > 0
\end{equation}
for some real constant $b$. 

When $\alpha = n$, it is exactly the complex Monge-Amp\`ere equation, which is strongly connected with complex geometry. In the fundamental work of Yau~\cite{Yau78} (see also \cite{Aubin78}), he proved the Calabi conjecture~\cite{Calabi56},~~\cite{Calabi57} by solving the complex Monge-Amp\`ere equation. Cao~\cite{Cao85} reproduced the result of Yau~\cite{Yau78} and Aubin~\cite{Aubin78} by K\"ahler--Ricci flow. Cherrier~\cite{Cherrier87}, Tosatti and Weinkove~\cite{TWv10a} independently extended the zero order estimate of Yau~\cite{Yau78} to Hermitian manifolds under the balanced condition, and then solved the complex Monge-Amp\`ere equations  on closed Hermitian manifolds by method of continuity. Later, Tosatti and Weinkove~\cite{TWv10b} successfully removed the balanced condition and extended the result to general Hermitian manifolds. Gill~\cite{Gill11} introduced the Chern--Ricci flow, and gave a parabolic proof for the result in \cite{TWv10b}.  

For $\alpha = 1$, equation~\eqref{pmach-elliptic-equation} was proposed by Donaldson~\cite{Donaldson99a} in connection with moment maps and is closely related to the Mabuchi energy \cite{Chen04}, \cite{Weinkove06}, \cite{SW08} for K\"ahler manifolds. It is well known that when $\chi$ and $\omega$ are K\"ahler, there is an invariant defined by
\begin{equation}
\label{pmach-kahler-constant}
 c = \frac{\int_M \chi^n}{\int_M \chi^{n - \a} \wedge \omega^\a}.
\end{equation}  
The equation was studied by Chen~\cite{Chen00b},  \cite{Chen04}, Weinkove~\cite{Weinkove04}, \cite{Weinkove06}, Song and Weinkove~\cite{SW08} using the $J$--flow where $\psi = c$. Their result was extended by Fang, Lai and Ma~\cite{FLM11} to all $1 \leq \alpha < n$. In \cite{FL12}, \cite{FL13}, Fang and Lai studied a class of geometric flows when $\psi = c$, which include  equation~\eqref{pmach-parabolic-equation}. It attracts our attention to generalize the results to general Hermitian manifolds or general $\psi$. Guan and the author~\cite{GSun12} studied the Dirichlet problem on general Hermitian manifolds with the assumption of subsolution. The author~\cite{Sun2013e} solved equation~\eqref{pmach-elliptic-equation} by method of continuity on closed Hermitian manifolds under the cone condition.


In the study, the sharp $C^2$ estimate is the key. We prove the following theorem.
\begin{theorem}
\label{pmach-int-thm-main}
Let $(M^n , \omega)$ be a closed Hermitian manifold of complex dimension $n$. Suppose that $\chi$ is a smooth real $(1,1)$ form satisfying $\chi \in \mathscr{C}_\a (\psi)$. Then there exists a long time solution $u$ to equation~\eqref{pmach-parabolic-equation}. Moreover, there are constants $C$ and $A$ such that
\begin{equation}
\label{pmach-int-thm-C2}
 \Delta u + \tr \chi \leq C e^{A (u - \inf_{M\times [0,t]} u)},
\end{equation}
where $C$, $A$ depend only on initial geometric data.
\end{theorem}

The $C^2$ estimate is much more improved than that in \cite{Sun2013p}, which is
\begin{equation}
	\Delta u + tr\chi \leq C e^{\left( e^{A(\sup_{M\times[0,t]} (u - \ul u) - \inf_{M\times[0,t]} (u - \ul u)) - e^{A (\sup_{M\times[0,t]} (u - \ul u) - (u-\ul u))}}\right)} .
\end{equation}
The improved $C^2$ estimate can help us to obtain the uniform {\em a priori} $C^0$ estimate if we have proper conditions. The $C^0$ is one of the most difficult estimates for differential equations on closed manifolds. In fact, the higher order estimates thus follow from the $C^0$ and $C^2$ estimates, Evans-Krylov theory~\cite{Evans82},~\cite{Krylov82} and Schauder estimates, which are quite standard procedures.

We do not impose strong condition on $\psi$, and consequently it is very likely that the flow $u$ itself does not converge. To discover some convergence property, it is necessary to normalize the solution. Let
\begin{equation}
\label{pmach-int-definition-normalization}
	\tilde u = u - \frac{\int_M u \omega^n}{\int \omega^n} .
\end{equation}

For general Hermitian manifolds, we have the following result.
\begin{theorem}
\label{pmach-int-theorem-convergence-Hermitian}
Under the assumption of Theorem~\ref{pmach-int-thm-main}, there exists a uniform constant $C$ such that for all time $t \geq 0$,
\begin{equation}
	\sup_{x\in M} u(x,t) - \inf_{x\in M} u(x,t) < C,
\end{equation}
given that
\begin{equation}
    \frac{\chi^n}{\chi^{n - \a} \wedge \omega^\a}\leq \psi .
\end{equation}
Then $\tilde u$ is $C^\infty$ convergent to a smooth function $\tilde u_\infty$. Moreover, there is a unique real number $b$ such that the pair $(\tilde u_\infty , b)$ solves equation~\eqref{pmach-elliptic-equation-actual}.
\end{theorem}

Should we have more knowledge of the manifold, it would be possible to obtain deeper results. When $\chi$ and $\omega$ are both K\"ahler, we are able to solve Donaldson's problem formerly proven by the flow method in \cite{SW08}, \cite{FLM11}. It is worth a mention that the former results are only for the special case $\psi = c$ while we only require that $\psi \geq c$.

It is noticeable that in \cite{Sun2013e}, the elliptic approach has no obstacle when $\psi\geq c$ while the former parabolic approaches have difficulties to treat non-constant $\psi$. It is natural to ask: is the flow method also able to treat $\psi \geq c$? After a study of $J$-functional~\cite{Chen00b}, we use the functional to normalize the velocity $\p_t u$ instead of the function $u$. The idea is quite simple and natural: we pull back the surface at every instantaneous moment.
\begin{theorem}
\label{pmach-int-theorem-convergence-Kahler}
Let $(M^n,\omega)$ be a closed K\"ahler manifold of complex dimension $n$ and $\chi$ is also K\"ahler. Suppose that $\chi \in \mathscr{C}_\a (\psi)$ and $\psi \geq c$ for all $x \in M$,
where $c$ is defined in \eqref{pmach-kahler-constant}. Then there exists a uniform constant $C$ such that for all time $t \geq 0$,
\begin{equation}
	\sup_{x\in M} u(x,t) - \inf_{x\in M} u(x,t) < C.
\end{equation}
Consequently, $\tilde u$ is $C^\infty$ convergent to a smooth function $\tilde u_\infty$. Moreover, there is a unique real number $b$ such that the pair $(\tilde u_\infty , b)$ solves equation~\eqref{pmach-elliptic-equation-actual}.

\end{theorem}

We organize the paper as follows. In section~\ref{pmach-pre}, we state and show some preliminary knowledge related to equation~\eqref{pmach-parabolic-equation}. In section~\ref{pmach-C2}, we establish the improved $C^2$ estimate. In section~\ref{pmach-long}, we study the long time existence of the solution flow. The $C^2$ estimate proven in section~\ref{pmach-C2} helps us to obtain the uniform $C^0$ estimate, and hence we show that higher order estimates are also uniform. In section~\ref{pmach-Harnack}, we give a proof for a Li-Yau type Harnack inequality. In section~\ref{pmach-convergence}, we apply the inequality to show that the time derivative of $\tilde u$ decays exponentially. As an immediate result, we show that $\tilde u$ convergences to a smooth function $\tilde u_\infty$ , which solves equation~\eqref{pmach-elliptic-equation-actual} for some $b$. In section~\ref{pmach-revisit}, we apply Theorem~\ref{pmach-int-theorem-convergence-Hermitian} to carrying out method of continuity, in order to avoid using some specific knowledge.

\bigskip

\section{Preliminaries}
\label{pmach-pre}
\setcounter{equation}{0}
\medskip

We shall follow \cite{GSun12} for notations. In particular, $g$ and $\nabla$ will denote the Riemannian metric and the corresponding Chern connection of $(M, \omega)$.  In local coordinates $z = (z^1 , \cdots , z^n)$ we have
\begin{equation}
	\label{pmach-formula-2}
    \left\{
    \begin{aligned}
        T^k_{ij} =& 
         \sum_l g^{k\bar l} \left(\f{\p g_{j\bar l}}{\p z^i} - \frac{\p g_{i\bar l}}{\p z^j}\right) \,,  \\
        R_{i\bar jk\bar l} =& 
         - \frac{\p^2 g_{k\bar l}}{\p z^i \p \bar z^j} + \sum_{p,q} g^{p\bar q} \f{\p g_{k\bar q}}{\p z^i} \f{\p g_{p\bar l}}{\p\bar z^j} .
    \end{aligned}
    \right.
\end{equation}

For convenince, we set
\begin{equation}
\label{int:definition-X}
X := \chi_u = \chi+\frac{\sqrt{-1}}{2}\partial\bar\partial u \,,
\end{equation}
and thus
\begin{equation}
\label{int:definition-X-coefficients}
X_{i\bar j} = \chi_{i\bar j} + \bpartial_j\p_i u\,.
\end{equation}
Also, we denote the coefficients of $X^{-1}$ by $X^{i\bar j}$. It is easy to see that
\begin{equation}
\label{pmach-X-covariant-1}
    	\overline{X_{i\bar jk}} = X_{j\bar i\bar k} \,.
\end{equation}
Assuming at the point $p$,  $g_{i\bar j} = g_{ij}$ and $X_{i\bar j}$ is diagonal in a specific chart.  
Therefore,
\begin{equation}
\label{pmach-formula-X-1}
\begin{split}
     	X_{i\bar ij\bar j} - X_{j\bar ji\bar i} &= R_{j\bar ji\bar i}X_{i\bar i}  -  R_{i\bar ij\bar j}X_{j\bar j} + 2 \mathfrak{Re} \Big\{\sum_p \overline{T^p_{ij}}X_{i\bar pj}\Big\} \\
	&\hspace{6em}- \sum_{p} T^p_{ij} \overline{T^p_{ij}} X_{p\bar p} - G_{i\bar ij\bar j},\\
\end{split}
\end{equation}
where
\begin{equation}
\label{pmach-formula-G-coefficient}
\begin{aligned}
    	G_{i\bar ij\bar j} &= \chi_{j\bar ji\bar i} - \chi_{i\bar ij\bar j} + \sum_p R_{j\bar ji\bar p}\chi_{p\bar i} -\sum_p R_{i\bar ij\bar p}\chi_{p\bar j}  \\
    	&\hspace{3em} + 2\mathfrak{Re}\Big\{\sum_p \overline{T^p_{ij}}\chi_{i\bar pj} \Big\} - \sum_{p,q}T^p_{ij}\overline{T^q_{ij}}\chi_{p\bar q}\,.
\end{aligned}
\end{equation}

Let $S_\a (\lambda)$ denote the $\a$-th elementary symmetric polynomial of $\lambda \in \bfR^n$,
\begin{equation}
	S_\a (\lambda) = \sum_{1 \leq i_1 < \cdots < i_\a \leq n} \lambda_{i_1} \cdots \lambda_{i_\a} \,.
\end{equation}
For a square matrix $A$, define $S_\a (A) = S_\a (\lambda(A))$ where $\lambda(A)$ denote the eigenvalues of $A$. Further, write $S_\a (X) = S_\a (\lambda_* (X))$ and $S_\a (X^{-1}) = S_\a (\lambda^* (X))$ where $\lambda_* (X)$ and $\lambda^* (X)$ denote the eigenvalues of a Hermitian matrix $X$ with respect to $\omega$ and to $\omega^{-1}$, respectively. Unless otherwise indicated we shall use $S_\alpha$ to denote $S_\alpha(X^{-1})$ when no possible confusion would occur.

In local coordinates, we can write equation~\eqref{pmach-parabolic-equation} in the form
\begin{equation}
	\frac{\p u}{\p t} = \ln S_n (\chi_u) - \ln S_{n - \a} (\chi_u) + \ln C^\a_n - \ln \psi .
\end{equation}
or equivalently,
\begin{equation}
\label{pmach-main-equation}
    	\frac{\partial u}{\partial t} = \ln C^\alpha_n - \ln \psi - \ln S_\alpha(\chi^{-1}_u).
\end{equation}

Differentiating this equation twice  at $p$ and applying the strong concavity of $S_\a$ as in \cite{GSun12} , we have
\begin{equation}
\label{pmach-formula-S-partial}
     	\partial_t\partial_l u= - \frac{\partial_l \psi}{\psi} + {S^{-1}_\a} \sum_i S_{\a -1;i} (X^{i\bi})^2 X_{i\bi l} ,
\end{equation}
and 
\begin{equation}
\label{pmach-formula-S-double-derivative}
\begin{aligned}
    	\partial_t\bar\partial_l\partial_l u &\leq  C  + S^{-2}_\a\sum_i \p_i S_\a \bpartial_i S_\a - S^{-1}_\a\sum_{i,j} S_{\alpha -1;i}  (X^{i\bar i})^2 X^{j\bar j}X_{j\bar i\bar l}X_{i\bar j l} \\
    	&\hspace{5em} +  S^{-1}_\a \sum_i S_{\alpha -1;i}(X^{i\bar i})^2 X_{i\bar il\bar l} .
\end{aligned}
\end{equation}

Since $\ul u \in C^2(M)$ and $\chi_{\ul u} >0$, 
\begin{equation}
    \epsilon\omega \leq \chi_{\ul u} \leq \epsilon^{-1} \omega
\end{equation}
for some $\epsilon > 0$.

Here we write down an alternative statement, which will be used later. It is a  key theorem in \cite{FLM11}.
\begin{theorem}
\label{pmach-theorem-alternative}
We have an alternative
	\begin{equation}
        \sum_i S_{\alpha -1;i} (X^{i\bar i})^2 (\chi_{i\bar i} + \bpartial_i\partial_i\ul u) \geq (1+\theta) \alpha \Big(\frac{\psi}{C^\alpha_n}\Big)^{1/\alpha}S^{1+1/\alpha}_\alpha
	\end{equation}
for some $\theta> 0$ , or
	\begin{equation*}
		w \leq C
	\end{equation*}
for some $C>0$ large enough and independent on $u$.
\end{theorem}

Moreover, by the maximum principle, $\frac{\p u}{\p t}$ reaches the extremal values at $t=0$, that is, on a maximal time interval $[0,T)$,
\begin{equation}
	\left|\frac{\p u}{\p t}\right| \leq \sup_{M\times\{0\}} \left|\frac{\p u }{\p t}\right|.
\end{equation}
As an immediate corollary, 
\begin{equation}
	\left|\frac{\p\tilde u}{\p t}\right| \leq 2 \sup_{M\times\{0\}} \left|\frac{\p u }{\p t}\right| ,
\end{equation}
and
\begin{equation}
	\left|\ln  \frac{\chi_u^n }{\chi_u^{n - \a} \wedge \omega^\a} - \ln \psi\right| \leq \sup_{M\times\{0\}}\left|\ln  \frac{\chi_u^n }{\chi_u^{n - \a} \wedge \omega^\a} - \ln \psi\right|.
\end{equation}
Therefore, the flow remains Hermitian at any time.
\bigskip

\section{The second order estimates}
\label{pmach-C2}
\setcounter{equation}{0}
\medskip

In this section we derive the second order estimate for admissible solutions.

\begin{proposition}
    \label{pmach-proposition-C2}
    Let $u\in C^4(M\times[0,T))$ be an admissible solution of equation~\eqref{pmach-parabolic-equation} and $w = \Delta u + tr\chi$. Then there are uniform positive constants $C$ and $A$ such that
\begin{equation}
\label{pmach-C2-1}
 w \leq C e^{A (u - \inf_{M\times [0,t]} u)},
\end{equation}
where $C$, $A$ depend only on geometric data.
\end{proposition}

\begin{proof}
Consider the function $ e^\phi w$ where $\phi$ is to be specified later. Suppose that $e^\phi w$ its maximum at some point $(p,t_0)\in M_t = M\times(0,t]$. Choose a local chart near $p$ such that $g_{i\bar j} = \delta_{ij}$ and $X_{i\bar j}$ is diagonal at $p$ when $t = t_0$. Therefore,  we have at the point $(p,t_0)$
\begin{equation}
\label{pmach-proposition-C2-max-1-1}
\frac{\partial_l w}{w} + \partial_l\phi = 0,
\end{equation}
\begin{equation} 
\label{pmach-propostition-C2-max-1-2}
\frac{\bar\partial_l w}{w} + \bar\partial_l\phi = 0,
\end{equation}
\begin{equation}
\label{pmach-proposition-C2-max-1-t}
\frac{\p_t w}{w} + \p_t \phi \geq 0
\end{equation}
and
\begin{equation}
\label{pmach-proposition-C2-max-2}
\frac{\bar\partial_l\partial_l w}{w} - \frac{\bar\partial_l w\partial_l w}{w^2} + \bar\partial_l\partial_l\phi \leq 0 .
\end{equation}

Applying $\frac{\p}{\p t} - S^{-1}_\a \sum_l S_{\a -1;l} (X^{l\bar l})^2 \p_l\bpartial_l $ to the function $e^\phi w$, 
\begin{equation}
\label{pmach-proposition-C2-inequality-max}
\begin{aligned}
	&\, \p_t\left(e^\phi w\right) - S^{-1}_\a \sum_l S_{\a -1;l} (X^{l\bar l})^2 \p_l\bpartial_l  \left(e^\phi w\right) \\
	=&\, e^\phi \left(w\p_t \phi + \p_t w -  S^{-1}_\a \sum_l S_{\a -1;l} (X^{l\bar l})^2 \left(w\p_l\bpartial_l  \phi  - w^{-1}  |\p_l w|^2  +  \p_l\bpartial_l w  \right) \right).
\end{aligned}
\end{equation}
This formula, with \eqref{pmach-proposition-C2-max-1-t} and \eqref{pmach-proposition-C2-max-2}, leads to that at $(p,t_0)$
\begin{equation}
\label{pmach-proposition-C2-inequality-max-1}
\begin{aligned}
    	0 &\leq w\partial_t \phi + \partial_t w -   S^{-1}_\a  \sum_l S_{\alpha -1;l} (X^{l\bar l})^2\left( w \bar\partial_l\partial_l \phi - w^{-1} |\partial_l w|^2 + \sum_{i}  X_{i\bar il\bar l} \right).
\end{aligned}
\end{equation}

By \eqref{pmach-formula-X-1} and \eqref{pmach-formula-S-double-derivative}, we have at $(p,t_0)$
\begin{equation}
\label{pmach-proposition-C2-inequality-max-2}
\begin{aligned}
	&\; \sum_l S_{\alpha -1;l} (X^{l\bar l})^2\left( w \bar\partial_l\partial_l \phi - w^{-1} |\partial_l w|^2 + \sum_{i}  X_{i\bar il\bar l} \right) \\
     	\geq&\;  w\sum_i S_{\a -1;i} (X^{i\bar i})^2 \bpartial_i\p_i \phi -  w^{-1}\sum_{i} S_{\a -1;i} (X^{i\bar i})^2 |\p_i w|^2 \\
     	&\; + \sum_{i,j} S_{\a -1;i} (X^{i\bar i})^2 \Big( - R_{j\bar ji\bar i}X_{i\bar i} + R_{i\bar ij\bar j}X_{j\bar j}  + G_{i\bar ij\bar j}\Big) \\
     	&\; - 2\sum_{i,j,l} S_{\a -1;i}(X^{i\bar i})^2  \fRe\{\ol{T^j_{il}}X_{i\bar j l}\} + \sum_{i,j,l} S_{\a -1;i} (X^{i\bar i})^2  T^j_{il} \ol{T^j_{il}}X_{j\bar j} \\
    	&\; + \sum_{i,j,l} S_{\a -1;i} (X^{i\bar i})^2 X^{j\bar j} X_{j\bar i\bar l} X_{i\bar jl} + S_\a \sum_i \p_t\bpartial_i\p_i u - C S_\a.
\end{aligned}
\end{equation}
As in \cite{Sun2013e}, we can simplify equation~\eqref{pmach-proposition-C2-inequality-max-2},
\begin{equation}
\label{pmach-proposition-C2-inequality-max-2-1}
\begin{aligned}
	&\; \sum_l S_{\alpha -1;l} (X^{l\bar l})^2\left( w \bar\partial_l\partial_l \phi - w^{-1} |\partial_l w|^2 + \sum_{i}  X_{i\bar il\bar l} \right) \\
    	\geq&\; S_\a \sum_i \p_t\bpartial_i\p_i u - C S_\a + \sum_{i,j} S_{\a -1;i} (X^{i\bar i})^2 \Big( - R_{j\bar ji\bar i}X_{i\bar i} + R_{i\bar ij\bar j}X_{j\bar j}  + G_{i\bar ij\bar j}\Big) \\
    	&\; + w\sum_i S_{\a -1;i} (X^{i\bar i})^2 \bpartial_i\p_i \phi + \frac{2}{w} \sum_{i,j} S_{\a - 1;i} (X^{i\bar i})^2 \mathfrak{Re}\{ \hat{T}^k_{ji} \chi_{k\bar j} {\bpartial_i w}\}  
\end{aligned}
\end{equation}
where $\hat{T}$ denotes the torsion with respect to the Hermitian metric $\chi$. In \cite{Sun2013e}, we just simply threw away the last term.

Now we apply a trick due to Phong and Sturm \cite{PhongSturm10} and use the following function
\begin{equation}
\label{pmach-proposition-C2-test-function}
    	\phi := - A ( u - \ul u) + \frac{1}{u - \ul u - \inf_{M_t} (u - \ul u) + 1} = - A ( u - \ul u) + E_1  .
\end{equation}
Without loss of genelarity, we can assume $A \gg 1$ throughout this paper. 

It is easy to see that
\begin{equation}
\label{pmach-proposition-C2-test-function-t-derivative}
	\p_t \phi = - A  \p_t u - \p_t u E^2_1 ,
\end{equation}
\begin{equation}
\label{pmach-proposition-C2-test-function-first-derivative}
	\p_i \phi = - A ( \p_i u - \p_i \ul u) - {(\p_i u - \p_i\ul u)}{E^2_1} 
\end{equation}
and
\begin{equation}
\label{pmach-proposition-C2-test-function-second-derivative}
\begin{aligned}
    	\bpartial_i \p_i \phi =& - A (\bpartial_i\p_i u - \bpartial_i\p_i \ul u) -  {(\bpartial_i\p_i u - \bpartial_i\p_i \ul u)}{E^2_1} + {2 |\p_i(u - \ul u)|^2}{E^3_1} .  
\end{aligned}
\end{equation}
The fourth term in \eqref{pmach-proposition-C2-inequality-max-2-1} turns to 
\begin{equation}
\begin{aligned}
     	 w \sum_i S_{\a -1;i} (X^{i\bar i})^2 \bpartial_i\p_i \phi =&\, 2 w \sum_i S_{\a -1;i} (X^{i\bar i})^2 { |\p_i(u - \ul u)|^2}{E^3_1} \\
     	 &- \Big( Aw + {w}{E^2_1}  \Big)\sum_i S_{\a -1;i} (X^{i\bar i})^2(\bpartial_i\p_i u-  \bpartial_i\p_i \ul u)    ;
\end{aligned}
\end{equation}
and the fifth term is
\begin{equation}
\begin{aligned}
    	 \frac{2}{w} \sum_{i,j} S_{\a - 1;i} &(X^{i\bar i})^2 \mathfrak{Re}\{ \hat{T}^k_{ji} \chi_{k\bar j} {\bpartial_i w}\}   =  - 2  \sum_{i,j} S_{\a - 1;i} (X^{i\bar i})^2 \mathfrak{Re}\{\hat{T}^k_{ji} \chi_{k\bar j} {\bpartial_i \phi}\} \\
    	\geq&\, - {w}{E^3_1} \sum_i S_{\a -1;i} (X^{i\bar i})^2 |\p_i (u - \ul u)|^2-  \frac{C A^2}{w E^3_1} \sum_i S_{\a -1;i} (X^{i\bar i})^2 .
\end{aligned}
\end{equation}
Therefore,
\begin{equation}
\label{pmach-proposition-C2-inequality-max-2-2}
\begin{aligned}
	&\; S^{-1}_\a \sum_l S_{\alpha -1;l} (X^{l\bar l})^2\left( w \bar\partial_l\partial_l \phi - w^{-1} |\partial_l w|^2 + \sum_{i}  X_{i\bar il\bar l} \right) \\
    	\geq& \left( A + {E^2_1} \right) w S^{-1}_\a \sum_i S_{\a -1;i} (X^{i\bar i})^2(\bpartial_i\p_i \ul u - \bpartial_i\p_i  u) - \frac{C A^2 }{w E^3_1} S^{-1}_\a\sum_i S_{\a -1;i} (X^{i\bar i})^2 \\
	& + S^{-1}_\a\sum_{i,j} S_{\a -1;i} (X^{i\bar i})^2 \Big( - R_{j\bar ji\bar i}X_{i\bar i} + R_{i\bar ij\bar j}X_{j\bar j}  + G_{i\bar ij\bar j}\Big) + \sum_i \p_t\bpartial_i\p_i u - C  .
\end{aligned}
\end{equation}
Note also that
\begin{equation}
\label{pmach-proposition-C2-inequality-max-1-1}
\begin{aligned}
	w\partial_t \phi + \partial_t w 
	&= - (A + E^2_1 ) w\ln \left(\frac{C^\a_n}{\psi S_\a}\right) +  \sum_i \bar\partial_i \partial_i \partial_t u .
\end{aligned}
\end{equation}
Combing \eqref{pmach-proposition-C2-inequality-max-1}, \eqref{pmach-proposition-C2-inequality-max-2-2} and \eqref{pmach-proposition-C2-inequality-max-1-1},
\begin{equation}
\begin{aligned}
	0 \leq&\,    - \left( A + E^2_1 \right) w \left( \ln\left(\frac{C^\a_n}{\psi S_\a}\right) + S^{-1}_\a \sum_i S_{\a -1;i} (X^{i\bar i})^2(\bpartial_i\p_i \ul u - \bpartial_i\p_i  u) \right) \\
	& + \frac{C A^2 }{w E^3_1} S^{-1}_\a \sum_i S_{\a -1;i} (X^{i\bar i})^2   + C (w + 1 )\ S^{-1}_\a \sum_{i,j} S_{\a -1;i} (X^{i\bar i})^2 + C .
\end{aligned}
\end{equation}

For $A \gg 1$ which is to be determined later, there are two cases in consideration:
(1) $w > A(u  - \ul u - \inf_{M_t}(u - \ul u) + 1)^{\frac{3}{2}} \geq A > N$, where $N$ is the crucial constant in Therorem~\ref{pmach-theorem-alternative};
(2) $w \leq A(u  - \ul u - \inf_{M_t}(u - \ul u) + 1)^{\frac{3}{2}}$ .

In the first case, by Theorem~\ref{pmach-theorem-alternative},
\begin{equation}
\begin{aligned}
 	&\, \ln \left(\frac{C^\a_n}{\psi S_\a}\right) + S^{-1}_\a \sum_i S_{\a -1;i} (X^{i\bar i})^2(\bpartial_i\p_i \ul u - \bpartial_i\p_i  u) \\
	\geq&\,  \ln \left(\frac{C^\a_n}{\psi S_\a}\right) + (1+\theta) \alpha \Big(\frac{\psi S_\a}{C^\alpha_n}\Big)^{\frac{1}{\alpha}} - \a \\
	\geq&\, \theta \alpha \left(\frac{\psi S_\a}{C^\alpha_n}\right)^{\frac{1}{\alpha}} 
\end{aligned}
\end{equation}
by observing the simple fact
\begin{equation}
    \ln x \leq x - 1, \qquad \forall x > 0.
\end{equation}
Alloting the extra positive $\theta$ properly, we have for $\delta > 0$ small enough,
\begin{equation}
	0 \leq   - A w \delta\left( 1 + S^{-1}_\a \sum_i S_{\a -1;i} (X^{i\bar i})^2 \right)   + C w \ S^{-1}_\a \sum_{i,j} S_{\a -1;i} (X^{i\bar i})^2 + C .
\end{equation}
This gives a bound $w \leq 1$ at $p$ if we pick a sufficiently large $A$, which contradicts the assumption $A \gg 1$.

In the second case, 
\begin{equation}
\begin{aligned}
    w e^\phi \leq w e^\phi |_p 
    &\leq A\left(u  - \ul u - \inf_{M_t}(u - \ul u) + 1\right)^{\frac{3}{2}} e^{ - A ( u - \ul u) + 1} \\
    &\leq A e^2 e^{- A  \inf_{M_t}(u - \ul u) } 
\end{aligned}
\end{equation}
and hence
\begin{equation}
\begin{aligned}
    	w &\leq A e^2 e^{ A ( u - \ul u) - E_1 - A  \inf_{M_t}(u - \ul u) } \\
	&\leq A e^2 e^{ A ( u - \ul u) - A  \inf_{M_t}(u - \ul u) } \leq C e^{A (u - \inf_{M_t} u)} .
\end{aligned}
\end{equation}

\end{proof}

The Evans-Krylov theory and Schauder estimates are quite standard, so we omit the proofs here. We refer readers to \cite{GL10}, \cite{TWv10a}, \cite{Sun2013e}, \cite{Gill11} and \cite{Sun2013p}.

\bigskip

\section{Long time existence and uniform estimates}
\label{pmach-long}
\setcounter{equation}{0}
\medskip

Since $\frac{\p u}{\p t}$ is bounded, in finite time we have $C^0$ estimate dependent on time $t$. By the Evans-Krylov theory and Schauder estimates, we can obtain time-dependent $C^\infty$ estimates. With a standard argument,  theses estimates are sufficient for us to prove the long time existence.

Nevertheless, in order to prove the convergence, we need uniform $C^0$ estimate, which is independent from the time $t$.

\begin{theorem}
\label{pmach-long-theorem-hermitian}
Under the assumption of Theorem~\ref{pmach-int-thm-main}, there exists a uniform constant $C$ depeding only on the initial geometric data such that
\begin{equation}
	\sup_M u(x,t) - \inf_M u(x,t) \leq C,
\end{equation}
given that
\begin{equation}
\label{pmach-long-theorem-hermitian-condition}
    \frac{\chi^n}{\chi^{n - \a} \wedge \omega^\a}\leq \psi .
\end{equation}
\end{theorem}
\begin{proof}

Following \cite{Weinkove06}, we prove the theorem by contradiction. If such a bound does not exist, we can choose a sequence of time points $t_i \rightarrow \infty$ such that
\begin{equation}
	\sup_M u(x,t_i) - \inf_M u(x,t_i) \rightarrow \infty,
\end{equation}

Note that
\begin{equation}
	\frac{\p u}{\p t} \leq 0
\end{equation}
by \eqref{pmach-long-theorem-hermitian-condition} and the maximum principle. Thus for $t > s \geq 0$,
\begin{equation}
\label{pmach-long-theorem-hermitian-decreasing-sup}
	\sup_M u(x , t) \leq \sup_M u(x , s)
\end{equation}
and
\begin{equation}
\label{pmach-long-theorem-hermitian-decreasing-inf}
	\inf_M u(x , t) \leq \inf_M u(x , s) .
\end{equation}
So
\begin{equation}
	\inf_M u(x , t_i) = \inf_{t\in [0,t_i]}\inf_M u(x , t) \rightarrow - \infty.
\end{equation}

By Proposition~\ref{pmach-proposition-C2},
\begin{equation}
\begin{aligned}
 	w(x, t_i) &\leq C e^{A \left( u(x,t_i) - \inf_{M\times [0,t_i]} u \right)} \\
	& = C e^{A \left( u(x,t_i) - \inf_{M}  u(x,t_i)\right)} .
\end{aligned}
\end{equation}
As shown in \cite{TWv10a}, the sharp $C^2$ estimate implies that
\begin{equation}
	\sup_M u(x,t_i) - \inf_M u(x,t_i) \leq C ,
\end{equation}
for some positive constant $C$, which is a contradiction.

\end{proof}

As a consequence, we also obtain uniform the $C^2$ estimate. By \eqref{pmach-long-theorem-hermitian-decreasing-inf}, it must be 
\begin{equation}
	 w(x,t) \leq C e^{A (u(x,t) - \inf_{M} u(x,t))} \leq  C e^{A (\sup_M u(x,t) - \inf_{M} u(x,t))} .
\end{equation}

Should we have more knowledge of the two metrics, there would be chances to obtain deeper results.
\begin{theorem}
\label{pmach-long-theorem-kahler}
Under the assumption of Theorem~\ref{pmach-int-theorem-convergence-Kahler}, there exists a uniform constant $C$ depeding only on the initial geometric data such that
\begin{equation}
	\sup_M u(x,t) - \inf_M u(x,t) \leq C,
\end{equation}
given that
\begin{equation}
\label{pmach-long-theorem-kahler-condition}
    \psi \geq c,
\end{equation}
where $c$ is a well known invariant defined by \eqref{pmach-kahler-constant}.
\end{theorem}


First of all, we need to extend the definition of $J$-functional~\cite{Chen00b} , which was done in \cite{FLM11}. Let $\mathcal{H}$ be the space of K\"ahler potentials
\begin{equation}
	\mathcal{H} := \{ u\in C^\infty(M) \;|\; \chi_u \in [\chi]^+\} .
\end{equation}
For any curve $v(s) \in \mathcal{H}$, we define the funtional $J_\a$ by
\begin{equation}
\label{pmach-J-functional-defition-derivative}
	\frac{d J_\a}{d s} = \int_M \frac{\p v}{\p s} \chi^{n - \a}_v \wedge \omega^\a .
\end{equation}
Then we have a formula for $J_\a$ of function $u$,
\begin{equation}
\label{pmach-J-functional-defition-formula}
J_\a (u) = \int^1_0 \int_M \frac{\p v}{\p s} \chi^{n - \a}_v \wedge \omega^\a ds,
\end{equation}
where $v(s)$ is an arbitrary path in $\mathcal{H}$ connecting $0$ and $u$. It is straighforward to verify that the functional is independent on the path. So we can restrict the integration to the straight line $v(s) = s u$ to calculate $J_\a(u)$  for any $u \in \mathcal{H}$.
\begin{equation}
\label{pmach-J-functional-calculation-line}
\begin{aligned}
	J_\a (u) &= \int^1_0 \int_M u \chi^{n - \a}_{s u}\wedge \omega^\a ds \\
	&= \int^1_0 \int_M u \left(s \chi_u + (1 - s)\chi\right)^{n - \a} \wedge \omega^\a ds \\
	&= \int^1_0 \int_M u \sum^{n - \a}_{i = 0} C^i_{n - \a} s^i (1 - s)^{n - \a - i} \chi^i_u \wedge \chi^{n - \a - i} \wedge \omega^\a ds \\
	&= \sum^{n - \a}_{i = 0} C^i_{n - \a} \int^1_0 s^i (1 - s)^{n - \a - i} ds \int_M u \chi^i_u \wedge \chi^{n - \a - i} \omega^\a \\
	&= \frac{1}{n - \a + 1} \sum^{n - \a}_{i = 0}  \int_M u \chi^i_u \wedge \chi^{n - \a - i} \omega^\a .
\end{aligned}
\end{equation}
Also, along the solution flow $u(x,t)$ to equation~\eqref{pmach-parabolic-equation},
\begin{equation}
\label{pmach-long-theorem-kahler-decreasing-J-functional}
\begin{aligned}
	\frac{d}{dt} J_\a (u) &= \int_M \frac{\p u}{\p t} \chi^{n - \a}_u \wedge \omega^\a \\
	&\leq \int_M \left(\ln \frac{\chi^n_u}{\chi^{n - \a}_u \wedge \omega^\a} - \ln\psi \right) \chi^{n - \a}_u \wedge \omega^\a \\
	&\leq \ln c\int_M \chi^{n - \a}_u \wedge \omega^\a  - \int_M \ln \psi \chi^{n - \a}_u \wedge \omega^\a \\
	&\leq 0 .
\end{aligned}
\end{equation}
The second inequality follows from Jesen's inequality. Thus, $J_\a (u)$ is decreasing and nonpositive along the solution flow.

Now we consider an arbitrary function flow $u(x,t)$ starting from $0$. Computing $J_\a$ on the flow, that is $v(x,s) = u(x,sT)$, it follows that
\begin{equation}
\label{pmach-J-functional-calculation-flow}
\begin{aligned}
	J_\a(u(T)) &= \int^1_0 \int_M T \frac{\p u}{\p t} (sT) \chi^{n - \a}_{u(sT)} \wedge \omega^\a ds \\
	&= \int^T_0 \int_M \frac{\p u}{\p t} \chi^{n - \a}_u \wedge \omega^\a dt \\
	&= \int^T_0 \frac{d J_\a}{d t} dt .
\end{aligned}
\end{equation}
For the solution flow $u(x,t)$ to equation~\eqref{pmach-parabolic-equation}, let 
\begin{equation}
\label{pmach-time-normalized-solution-defintion}
\hat u = u - \frac{J_\a (u)}{\int_M \chi^{n - \a}\wedge \omega^\a} .
\end{equation}

\begin{lemma}
\label{pmach-long-lemma-1}
\begin{equation}
\label{pmach-long-lemma-1-inequality}
	0 \leq \sup_M \hat u(x , t) \leq - C_1 \inf_M \hat u(x , t) + C_2 .
\end{equation}
\end{lemma}
\begin{proof}

We shall use the functional $J_\a$. By \eqref{pmach-J-functional-calculation-flow},
\begin{equation}
\label{pmach-long-lemma-1-J-flow}
\begin{aligned}
	J_\a (\hat u(T)) &= \int^T_0 \int_M \frac{\p \hat u}{\p t} \chi^{n - \a}_u \wedge \omega^\a dt \\
	&= \int^T_0 \int_M \left(\frac{\p u}{\p t} - \frac{1}{\int_M \chi^{n - \a}\wedge\omega^\a}\frac{d J_\a}{dt}\right) \chi^{n-\a}_u \wedge \omega^\a dt \\
	&= 0
\end{aligned}
\end{equation}
But according to \eqref{pmach-J-functional-calculation-line}, we have
\begin{equation}
\label{pmach-long-lemma-1-J-line}
	\frac{1}{n - \a + 1} \sum^{n - \a}_{i = 0}  \int_M \hat u \chi^i_{\hat u} \wedge \chi^{n - \a - i}  \wedge \omega^\a = 0. 
\end{equation}
The first inequality in \eqref{pmach-long-lemma-1-inequality} then follows from \eqref{pmach-long-lemma-1-J-line}.

Rewriting \eqref{pmach-long-lemma-1-J-line},
\begin{equation}
	  \int_M \hat u  \chi^{n - \a }  \wedge \omega^\a = - \sum^{n - \a}_{i = 1}  \int_M \hat u \chi^i_{\hat u} \wedge \chi^{n - \a - i}  \wedge \omega^\a. 
\end{equation}
Let $C_1$ be a positive constant such that
\begin{equation}
	\omega^n \leq C_1 \chi^{n - \a} \wedge \omega^\a .
\end{equation}
Then
\begin{equation}
\label{pmach-long-lemma-1-inequality-1}
\begin{aligned}
	\int_M \hat u \omega^n &= \int_M \left(\hat u - \inf_M \hat u\right) \omega^n + \int_M \inf_M \hat u \omega^n \\
	&\leq C_1 \int_M \left(\hat u - \inf_M \hat u\right) \chi^{n - \a} \wedge \omega^\a + \inf_M \hat u \int_M \omega^n \\
	&\leq - C_1 \sum^{n - \a}_{i = 1}  \int_M \hat u \chi^i_{\hat u} \wedge \chi^{n - \a - i}  \wedge \omega^\a \\
	&\hspace{5em}+ \inf_M \hat u \left(\int_M \omega^n - C_1 \int_M \chi^{n - \a}\wedge\omega^\a\right) \\
	&= - C_1 \sum^{n - \a}_{i = 1}  \int_M \left(\hat u - \inf_M \hat u\right) \chi^i_{\hat u} \wedge \chi^{n - \a - i}  \wedge \omega^\a \\
	&\hspace{5em}+ \inf_M \hat u \left(\int_M \omega^n -  (n - \a + 1)C_1 \int_M  \chi^{n - \a }  \wedge \omega^\a\right)  \\
	&\leq \inf_M \hat u \left(\int_M \omega^n -  (n - \a + 1)C_1 \int_M  \chi^{n - \a }  \wedge \omega^\a\right) .
\end{aligned}
\end{equation}
The second inequality \eqref{pmach-long-lemma-1-inequality} then follows from \eqref{pmach-long-lemma-1-inequality-1}, the fact that $\Delta_\omega u > - tr_\omega \chi$ and the lower bound of the Green's function of $\omega$ (see Yau~\cite{Yau78}).

\end{proof}

It remains to prove Theorem~\ref{pmach-long-theorem-kahler}.
\begin{proof}[Proof of Theorem~\ref{pmach-long-theorem-kahler}]
From Lemma~\ref{pmach-long-lemma-1} and the simple fact that
\begin{equation}
\sup_M u(x,t) - \inf_M u(x,t) = \sup_M \hat u(x,t) - \inf_M \hat u(x,t),
\end{equation}
it suffices to prove a lower bound for $\inf_M \hat u(x , t)$. If such a lower bound does not exist, then we can choose a sequence of time points $t_i \rightarrow \infty$ such that
\begin{equation}
\label{pmach-long-theorem-kahler-contradiction-1}
	\inf_M \hat u(x , t_i) = \inf_{t\in[0,t_i]} \inf_M \hat u(x, t)
\end{equation}
and
\begin{equation}
\label{pmach-long-theorem-kahler-contradiction-2}
	\inf_M \hat u(x , t_i) \rightarrow -\infty .
\end{equation}
Then for any $t \leq t_i$,
\begin{equation}
\label{pmach-long-theorem-kahler-decreasing-inf}
	\inf_M u(x , t_i) - \inf_M u(x , t) \leq \frac{J_\a (u(t_i))}{\int_M \chi^{n - \a}\wedge \omega^\a} - \frac{J_\a (u(t))}{\int_M \chi^{n - \a}\wedge \omega^\a} \leq 0 .
\end{equation}
With a similar argument, we also obtain
\begin{equation}
	u (x , t) \leq \hat u(x,t) .
\end{equation}
So we have
\begin{equation}
	\inf_M u(x , t_i) = \inf_{t\in[0,t_i]} \inf_M u(x, t)
\end{equation}
and
\begin{equation}
	\inf_M u(x , t_i) \rightarrow -\infty .
\end{equation}
By Proposition~\ref{pmach-proposition-C2},
\begin{equation}
\begin{aligned}
 	w(x, t_i) &\leq C e^{A \left( u(x,t_i) - \inf_{M\times [0,t_i]} u \right)} \\
	& = C e^{A \left( u(x,t_i) - \inf_{M}  u(x,t_i)\right)} .
\end{aligned}
\end{equation}
As shown in \cite{TWv10a}, the sharp $C^2$ estimate implies that
\begin{equation}
	\sup_M \hat u(x,t_i) - \inf_M \hat u(x,t_i) = \sup_M u(x,t_i) - \inf_M u(x,t_i) \leq C ,
\end{equation}
for some positive constant $C$, which is a contradiction since $\sup_M \hat u(x , t_i)$ is bounded from below by zero.

\end{proof}

For some $t_0 \in [0,t]$, $\inf_M u(x,t_0) = \inf_{M\times[0,t]} u(x,s)$ and hence
\begin{equation}
\begin{aligned}
	 u(x,t) - \inf_{M\times [0,t]} u(x,s) &= u(x,t) - \inf_{M} u(x,t_0) \\
	 &= \hat u(x,t) - \inf_{M} u(x,t_0) + \frac{J_\a (u(t))}{\int_M \chi^{n-\a} \wedge \omega^\a} \\
	 &\leq \hat u(x,t) - \inf_{M} \hat u(x,t_0) .
\end{aligned}
\end{equation}
The proof of Theorem~\ref{pmach-long-theorem-kahler} tells us that $\hat u(x,t)$ is uniformly bounded on $M\times[0,\infty)$. Consequently,
\begin{equation}
	 w(x,t) \leq  C e^{A (\sup_M \hat u(x,t) - \inf_{M} \hat u(x,t_0))} < C .
\end{equation}

Now we should use Evan-Krylov theory and Schauder estimate again to obtain uniform higher order estimates.

\bigskip

\section{The Harnack inequality}
\label{pmach-Harnack}
\setcounter{equation}{0}
\medskip

In this section, we prove the Harnack inequality. The arguments of Gill~\cite{Gill11} can be applied here. For completeness, we include the proof.

Set $F(u) := \frac{S_n (\chi_u)}{S_{n - \a} (\chi_u)}$.  We define a new Hermitian metric corresponding to $u$,
\begin{equation}
\label{pmach-defintion-new-metric}
    	G := F(u)\sum_{i,j} F_{i\bar j}(u) dz^i\wedge d\bar z^j ,
\end{equation}
and thus
\begin{equation}
\left\{
\begin{aligned}
	G_{i\bar j} &= F(u)F_{i\bar j}(u) ,\\
	G^{i\bar j} &= F^{-1}(u) F^{i\bar j},
\end{aligned}
\right.
\end{equation}
where $F^{i\bar j} := \frac{\p F}{\p u_{i\bar j}}$ and $[F_{i\bar j}]_{n\times n}$ is the inverse of $[F^{i\bar j}]_{n\times n}$.

Constructing the Hermitian metric is a key technique in \cite{Sun2013e} to carry out method of continuity, while the metric is quite natural in the study of the parabolic flow, especially when we study the properties of higher order derivatives, such as the Harnack inequality, Evans-Krylov theory and Schauder estimate. 

Let $\varphi$ be a positive function on $M$. We consider the linearized parabolic equation
\begin{equation}
	\p_t\varphi = \sum_{i,j} G^{i\bar j} \bpartial_j\p_i\varphi .
\end{equation}

Define $f = \ln \varphi$ and $H = t(|\p f|_G^2 - \beta \p_t f )$ where $1 < \beta < 2$. Here
\begin{equation}
	|\p f|_G^2 = G(\p f, \bpartial f) =  \sum_{i, j} G^{i\bar j} \p_i f \bpartial_j f .
\end{equation}
We also denote
\begin{equation}
	\left<X , Y\right>_G = G(X , \bar Y) = \sum_{i,j} G^{i\bar j} X_i Y_{\bar j} .
\end{equation}
Then
\begin{equation}
\label{pmach-harnack-equality-f-1}
\begin{aligned}
	&\; \p_t f - \sum_{i,j} G^{i\bar j} \bpartial_j\p_i f \\
	=&\; \frac{\p_t \varphi}{\varphi} - \frac{1}{\varphi} \sum_{i,j} G^{i\bar j} \bpartial_j\p_i \varphi + \frac{1}{\varphi^2} \sum_{i,j} G^{i\bar j} \p_i f \bpartial_j f \\
	=&\; |\p f|_G^2 .
\end{aligned}
\end{equation}

\begin{lemma}
\label{pmach-Harnack-lemma-1}
There are constants $C_1$ and $C_2$ depending only on the bounds of the metric $G$ such that for $t > 0$, we have
\begin{equation}
\begin{aligned}
	\p_t H - \sum_{k,l} G^{k\bar l} \bpartial_l\p_k H &\leq - \frac{t}{2n} (|\p f|^2_G - \p_t f)^2 + 2\mathfrak{Re} \left<\p f , \p H\right>_G \\
	& \hspace{1em} + (|\p f|^2_G - \beta\p_t f) + C_1 t |\p f|^2_G + C_2 t .
\end{aligned}
\end{equation}

\end{lemma}
\begin{proof}

By \eqref{pmach-harnack-equality-f-1},
\begin{equation}
\label{pmach-Harnack-lemma-1-1}
\begin{aligned}
	H &= t(|\p f|_G^2 - \beta \p_t f ) \\
	&= t \p_t f - t \sum_{i,j} G^{i\bar j} \bpartial_j\p_i f - t \beta \p_t f \\
	&= -t \sum_{i,j} G^{i\bar j} \bpartial_j\p_i f - t(\beta - 1)\p_t f .
\end{aligned}
\end{equation}
Then, 
\begin{equation}
\label{pmach-Harnack-lemma-1-2}
	\sum_{i,j} G^{i\bar j} \bpartial_j\p_i f = - \frac{1}{t} H - (\beta - 1)\p_t f ,
\end{equation}
and consequently
\begin{equation}
\label{pmach-Harnack-lemma-1-3}
	\frac{\p}{\p t} \left( \sum_{i,j} G^{i\bar j}\bpartial_j\p_i f\right) = \frac{1}{t^2} H - \frac{1}{t} \p_t H - (\beta - 1)\p^2_t f .
\end{equation}

Direct calculation shows that
\begin{equation}
\label{pmach-Harnack-lemma-1-4}
\begin{aligned}
	\p_t H &= |\p f|^2_G - \beta\p_t f + t \p_t \Big(\sum_{i,j} G^{i\bar j} \p_i f\bpartial_j f\Big) -\beta t \p^2_t f \\
	&= |\p f|^2_G - \beta\p_t f + t \sum_{i,j} \p_t G^{i\bar j} \p_i f\bpartial_j f + 2t \mathfrak{Re} \left<\p f , \p \p_t f\right>_G -\beta t \p^2_t f ,
\end{aligned}
\end{equation}
and hence
\begin{equation}
\label{pmach-Harnack-lemma-1-5}
	- 2 \mathfrak{Re} \left<\p f , \p \p_t f\right>_G = - \frac{1}{t} \p_t H + \frac{1}{t}|\p f|^2_G - \frac{\beta}{t}\p_t f +  \sum_{i,j} \p_t G^{i\bar j} \p_i f\bpartial_j f  -\beta  \p^2_t f .
\end{equation}
Also,
\begin{equation}
\label{pmach-Harnack-lemma-1-6}
\begin{aligned}
	\sum_{k,l} G^{k\bar l} \bpartial_l\p_k H &= t \sum_{k,l} G^{k\bar l} \Big( \sum_{i,j} \bpartial_l\p_kG^{i\bar j} \p_i f \bpartial_j f + \sum_{i,j} \p_k G^{i\bar j} \bpartial_l \p_i f \bpartial_j f \\
	& + \sum_{i,j} \p_k G^{i\bar j} \p_i f \bpartial_l\bpartial_j f + \sum_{i,j} \bpartial_l G^{i\bar j} \p_k\p_i f \bpartial_j f + \sum_{i,j} \bpartial_l G^{i\bar j} \p_i f \p_k\bpartial_j f \\
	& + \sum_{i,j} G^{i\bar j} \bpartial_l\p_k\p_i f \bpartial_j f + \sum_{i,j} G^{i\bar j} \p_k\p_i f \bpartial_l\bpartial_j f  + \sum_{i,j} G^{i\bar j} \bpartial_l\p_i f \p_k\bpartial_j f \\
	& + \sum_{i,j} G^{i\bar j} \p_i f \bpartial_l\p_k \bpartial_j f - \beta \bpartial_l\p_k\p_t f\Big)
\end{aligned}
\end{equation}

Now we control all the terms in \eqref{pmach-Harnack-lemma-1-6} by the bounds obtained in the previous sections. We bound the first term by
\begin{equation}
\label{pmach-Harnack-lemma-1-7}
	\left|\sum_{i,j,k,l} G^{k\bar l} \bpartial_l\p_k G^{i\bar j} \p_i f\bpartial_j f \right| \leq C_1 \left|\p f\right|^2_G.
\end{equation}
We bound the second, third, fourth and fifth terms by
\begin{equation}
	\left|\sum_{i,j,k,l} G^{k\bar l} \p_k G^{i\bar j} \bpartial_l \p_i f \bpartial_j f \right| \leq \frac{C_2}{\epsilon} \left|\p f\right|^2_G + \epsilon \left|\p\bpartial f\right|^2_G ,
\end{equation}
\begin{equation}
\label{pmach-Harnack-lemma-1-8}
	\left|\sum_{i,j,k,l} G^{k\bar l}\p_k G^{i\bar j} \p_i f \bpartial_l\bpartial_j f \right| \leq \frac{C_2}{\epsilon} \left|\p f\right|^2_G + \epsilon \left|D^2 f\right|^2_G ,
\end{equation}
\begin{equation}
\label{pmach-Harnack-lemma-1-9}
	\left|\sum_{i,j,k,l} G^{k\bar l} \bpartial_l G^{i\bar j} \p_k\p_i f \bpartial_j f\right| \leq \frac{C_2}{\epsilon} \left|\p f\right|^2_G + \epsilon \left|\p\bpartial f\right|^2_G ,
\end{equation}
and
\begin{equation}
\label{pmach-Harnack-lemma-1-10}
	\left|\sum_{i,j,k,l} G^{k\bar l}\bpartial_l G^{i\bar j} \p_i f \p_k\bpartial_j f\right| \leq \frac{C_2}{\epsilon} \left|\p f\right|^2_G + \epsilon \left|D^2 f\right|^2_G ,
\end{equation}
where
\begin{equation}
\label{pmach-Harnack-lemma-1-11}
	\left|\p\bpartial f\right|^2_G = \sum_{i,j,k,l} G^{k\bar l} G^{i\bar j} \bpartial_l\p_i f \p_k\bpartial_j f, \qquad \left|D^2 f\right|^2_G = \sum_{i,j,k,l} G^{k\bar l} G^{i\bar j} \p_k\p_i f \bpartial_l\bpartial_j f .
\end{equation}
It is obvious that the seventh and eighth terms are
\begin{equation}
\label{pmach-Harnack-lemma-1-12}
	\sum_{i,j,k,l} G^{k\bar l} G^{i\bar j} \p_k\p_i f \bpartial_l\bpartial_j f  = \left|D^2 f\right|^2_G
\end{equation}
and
\begin{equation}
\label{pmach-Harnack-lemma-1-13}
	\sum_{i,j,k,l} G^{k\bar l} G^{i\bar j} \bpartial_l\p_i f \p_k\bpartial_j f  = \left|\p\bpartial f\right|^2_G .
\end{equation}
From \eqref{pmach-Harnack-lemma-1-2} and \eqref{pmach-Harnack-lemma-1-5}, the sixth and nineth terms together give
\begin{equation}
\label{pmach-Harnack-lemma-1-14}
\begin{aligned}
	&\,\sum_{i,j,k,l} G^{k\bar l} G^{i\bar j} \left(\bpartial_l\p_k\p_i f \bpartial_j f + \p_i f \bpartial_l\p_k \bpartial_j f\right) \\
	\geq&\, 2 \mathfrak{Re} \left<\p f , \p \left(\sum_{k,l} G^{k\bar l}\bpartial_l\p_k f\right)\right>_G - \frac{C_3}{\epsilon}\left|\p f\right|^2_G - \epsilon |\p\bpartial f|^2_G \\
	=&\, -\frac{2}{t} \mathfrak{Re} \left<\p f , \p H\right>_G  - \frac{\beta - 1}{t} \p_t H + \frac{\beta - 1}{t}|\p f|^2_G - \frac{\beta(\beta - 1)}{t}\p_t f \\
	&\hspace{2em} +  (\beta - 1) \sum_{i,j} \p_t G^{i\bar j} \p_i f\bpartial_j f  -\beta(\beta - 1)  \p^2_t f - \frac{C_3}{\epsilon}\left|\p f\right|^2_G - \epsilon |\p\bpartial f|^2_G \\
	\geq&\, -\frac{2}{t} \mathfrak{Re} \left<\p f , \p H\right>_G  - \frac{\beta - 1}{t} \p_t H + \frac{\beta - 1}{t}|\p f|^2_G - \frac{\beta(\beta - 1)}{t}\p_t f \\
	&\hspace{2em} - C_4 \left|\p f\right|^2_G -\beta(\beta - 1)  \p^2_t f - \frac{C_3}{\epsilon}\left|\p f\right|^2_G - \epsilon |\p\bpartial f|^2_G .
\end{aligned}
\end{equation}
By  \eqref{pmach-Harnack-lemma-1-3}, the last term gives
\begin{equation}
\label{pmach-Harnack-lemma-1-15}
\begin{aligned}
	- \beta \sum_{k,l} G^{k\bar l}\bpartial_l\p_k\p_t f &= \beta \sum_{k,l} \p_t G^{k\bar l} \bpartial_l\p_k f - \beta \frac{\p}{\p t} \left(\sum_{k,l} G^{k\bar l}\bpartial_l\p_k f\right) \\
	&\geq - \frac{C_5}{\epsilon} - \epsilon \left|\p\bpartial f\right|^2_G - \frac{\beta}{t^2} H + \frac{\beta}{t} \p_t H + \beta (\beta - 1)\p^2_t f .
\end{aligned}
\end{equation}

Substituting\eqref{pmach-Harnack-lemma-1-7} --  \eqref{pmach-Harnack-lemma-1-10}, \eqref{pmach-Harnack-lemma-1-12} -- \eqref{pmach-Harnack-lemma-1-15} into \eqref{pmach-Harnack-lemma-1-6}, we have
\begin{equation}
\label{pmach-Harnack-lemma-1-16}
\begin{aligned}
	\sum_{k,l} G^{k\bar l} \bpartial_l\p_k H 
	&\geq \p_t H -2 \mathfrak{Re} \left<\p f , \p H\right>_G  - \left(\left|\p f\right|^2_G - \beta \p_t f\right) + t (1 - 4\epsilon) \left|\p\bpartial f\right|^2_G\\
	&\hspace{1em} - t \left(C_1 + C_4 + \frac{4 C_2}{\epsilon} + \frac{C_3}{\epsilon}\right)\left|\p f\right|^2_G + t (1 - 2\epsilon) \left|D^2 f\right|^2_G - \frac{t C_5}{\epsilon} \\
	&\geq \p_t H -2 \mathfrak{Re} \left<\p f , \p H\right>_G  - \left(\left|\p f\right|^2_G - \beta \p_t f\right) + \frac{t}{2} \left|\p\bpartial f\right|^2_G\\
	&\hspace{1em} - C t \left|\p f\right|^2_G - Ct
\end{aligned}
\end{equation}
when $\epsilon > 0$ is small enough. Applying the Schwarz inequality,
\begin{equation}
\label{pmach-Harnack-lemma-1-17}
\begin{aligned}
	\left|\p\bpartial f\right|^2_G \geq \frac{\left(\sum_{k,l} G^{k\bar l} \bpartial_l\p_k f\right)^2}{n} = \frac{\left(\p_t f - \left|\p f\right|^2_G\right)^2}{n} .
\end{aligned}
\end{equation}
Therefore,
\begin{equation}
\label{pmach-Harnack-lemma-1-18}
\begin{aligned}
	\sum_{k,l} G^{k\bar l} \bpartial_l\p_k H - \p_t H \geq&\,  -2 \mathfrak{Re} \left<\p f , \p H\right>_G  - \left(\left|\p f\right|^2_G - \beta \p_t f\right) \\
	&\,+ \frac{t}{2n} \left(\p_t f - \left|\p f\right|^2_G\right)^2 - C t \left|\p f\right|^2_G - C t .
\end{aligned}
\end{equation}

\end{proof}

\begin{lemma}
\label{pmach-Harnack-lemma-2}
There exist uniform constants $C_1$ and $C_2$ such that for all $t > 0$,
\begin{equation}
	\left|\p f\right|^2_G - \beta \p_t f \leq C_1 + \frac{C_2}{t}.
\end{equation}

\end{lemma}

\begin{proof}
Fix $T > 0$ and suppose that $H$ attains its maximum at $(p , t_0)$ in $M \times (0 , T]$. Then at the point $(p , t_0)$,
\begin{equation}
\label{pmach-Harnack-lemma-2-1}
	0 \geq  - \left(\left|\p f\right|^2_G - \beta \p_t f\right) + \frac{t_0}{2n} \left(\p_t f - \left|\p f\right|^2_G\right)^2 - C_1 t_0 \left|\p f\right|^2_G - C_2 t_0 .
\end{equation}

If $\p_t f (p,t_0) \geq 0$,
\begin{equation}
\label{pmach-Harnack-lemma-2-2}
\begin{aligned}
	0 &\geq  - \left(\left|\p f\right|^2_G - \p_t f\right) + \frac{t_0}{2n} \left(\p_t f - \left|\p f\right|^2_G\right)^2 - C_1 t_0 \left|\p f\right|^2_G - C_2 t_0 \\
	&=  \frac{t_0}{2n}\left(\left|\p f\right|^2_G - \p_t f\right) \left( \p_t f - \left|\p f\right|^2_G- \frac{2n}{t_0}\right) - C_1 t_0 \left|\p f\right|^2_G - C_2 t_0.
\end{aligned}
\end{equation}
Hence,
\begin{equation}
\label{pmach-Harnack-lemma-2-3}
	\left|\p f\right|^2_G - \p_t f \leq C_3 \left|\p f\right|_G + C_4 + \frac{C_5}{t_0} \leq \left(1 - \frac{1}{\beta}\right) \left|\p f\right|^2_G + C_6 + \frac{C_5}{t_0} .
\end{equation}
That is,
\begin{equation}
\label{pmach-Harnack-lemma-2-4}
	\left|\p f\right|^2_G - \beta \p_t f \le C_7 + \frac{C_8}{t_0}.
\end{equation}

If $\p_t f (p,t_0) < 0$,
\begin{equation}
\label{pmach-Harnack-lemma-2-5}
\begin{aligned}
	0 &\geq - \left(\left|\p f\right|^2_G - \beta \p_t f\right) + \frac{t_0}{2n} \left(\p_t f - \left|\p f\right|^2_G\right)^2 - C_1 t_0 \left|\p f\right|^2_G - C_2 t_0 \\
	&\geq - \left|\p f\right|^2_G + \beta \p_t f + \frac{t_0}{2n} \left|\p f\right|^4_G - C_1 t_0 \left|\p f\right|^2_G - C_2 t_0 \\
	&= \left|\p f\right|^2_G \left(-1 - C_1 t_0 + \frac{t_0}{2n} \left|\p f\right|^2_G \right) + \beta \p_t f - C_2 t_0 .
\end{aligned}
\end{equation}
That is,
\begin{equation}
\label{pmach-Harnack-lemma-2-6}
	\left|\p f\right|^2_G \left(-\frac{1}{t_0} - C_1  + \frac{1}{2n} \left|\p f\right|^2_G \right) \leq - \frac{\beta}{t_0} \p_t f + C_2 .
\end{equation}
Hence,
\begin{equation}
\label{pmach-Harnack-lemma-2-7}
	\left|\p f\right|^2_G \leq C_3 + \frac{C_4}{t_0} - \frac{1}{2} \p_t f .
\end{equation}
So
\begin{equation}
\label{pmach-Harnack-lemma-2-8}
\begin{aligned}
	0 &\geq - \left(\left|\p f\right|^2_G - \beta \p_t f\right) + \frac{t_0}{2n} \left(\p_t f - \left|\p f\right|^2_G\right)^2 - C_1 t_0 \left|\p f\right|^2_G - C_2 t_0 \\
	&\geq - \left|\p f\right|^2_G + \beta \p_t f + \frac{t_0}{2n} \left(\p_t f \right)^2 - C_1 t_0 \left|\p f\right|^2_G - C_2 t_0 .
\end{aligned}
\end{equation}
By factoring the terms, we obtain
\begin{equation}
\label{pmach-Harnack-lemma-2-9}
 	\left(-\p_t f\right) \left( - \frac{\beta}{t_0} - \frac{1}{2n} \p_t f \right)  \leq \frac{1}{t_0} \left|\p f\right|^2_G + C_1 \left|\p f\right|^2_G + C_2 .
\end{equation}
Thus
\begin{equation}
\label{pmach-Harnack-lemma-2-10}
\begin{aligned}
	- \p_t f &\leq C_5 + \frac{C_6}{t_0} + \frac{1}{2} \left|\p f\right|^2_G \\
	&\leq C_5 + \frac{C_6}{t_0} + \frac{C_3}{2} + \frac{C_4}{2 t_0} - \frac{1}{4}\p_t f
\end{aligned}
\end{equation}
and hence
\begin{equation}
\label{pmach-Harnack-lemma-2-11}
	- \p_t f \leq C_7 + \frac{C_8}{t_0} .
\end{equation}
Similarly
\begin{equation}
\label{pmach-Harnack-lemma-2-12}
\begin{aligned}
	\left|\p f\right|^2_G &\leq C_3 + \frac{C_4}{t_0} - \frac{1}{2} \p_t f \\
	&\leq C_3 + \frac{C_4}{t_0} + \frac{C_5}{2} + \frac{C_6}{2t_0} + \frac{1}{4} \left|\p f\right|^2_G
\end{aligned}
\end{equation}
and thus
\begin{equation}
\label{pmach-Harnack-lemma-2-13}
	\left|\p f\right|^2_G \leq C_9 + \frac{C_{10}}{t_0} .
\end{equation}
Combining \eqref{pmach-Harnack-lemma-2-11} and \eqref{pmach-Harnack-lemma-2-13},
\begin{equation}
\label{pmach-Harnack-lemma-2-14}
	\left|\p f\right|^2_G - \beta \p_t f \leq C_{11} + \frac{C_{12}}{t_0} .
\end{equation}

Therefore,
\begin{equation}
\label{pmach-Harnack-lemma-2-15}
\begin{aligned}
	H (x,T) &\leq H(p , t_0) \\
	&= t_0 \left(\left|\p f\right|^2_G(p , t_0) - \beta \p_t f (p , t_0)\right) \\
	&\leq \beta \left(C_1 t_0 + C_2\right) \\
	&\leq \beta \left(C_1 T + C_2\right) .
\end{aligned}
\end{equation}

\end{proof}

\begin{lemma}
\label{pmach-Harnack-lemma-3}
$\forall$ $0 < t_1 < t_2$,
\begin{equation}
	\sup_{x\in M} \varphi(x,t_1) \leq \inf_{x\in M}  \varphi(x,t_2) \Big(\frac{t_2}{t_1}\Big)^{C_2} e^{\frac{C_3}{t_2 - t_1} + C_1(t_2 - t_1)}
\end{equation}
where $C_1$, $C_2$ and $C_3$ are uniform constants.
\end{lemma}
\begin{proof}
Let $x, y \in M$, and define $\gamma$ to be the minimal geodesic (with respect ot the initial metric $\omega$) with $\gamma(0) = y$ and $\gamma(1) = x$. Define a path $\zeta:\; [0,1] \rightarrow M\times [t_1,t_2]$ by $\zeta(s) = \left(\gamma(s) , (1-s)t_2 + s t_1\right)$. Then using Lemma~\ref{pmach-Harnack-lemma-2},
\begin{equation}
\begin{aligned}
	\ln \frac{\varphi(x_1, t_1 )}{\varphi(y_2 ,t_2)} &= \int^1_0 \frac{d}{ds} f(\zeta(s))ds \\
	&= \int^1_0 \Big[2\mathfrak{Re} \left<\dot{\gamma}, \p f\right> - \left(t_2 - t_1\right) \p_t f\Big]ds \\
	&\leq \int^1_0 \left[2 |\dot{\gamma}|_G \left|\p f\right|_G + \left(t_2 - t_1\right) \left(C_1 + \frac{C_2}{t} - \frac{1}{\beta} \left|\p f\right|^2_G\right)\right] ds \\
	&\leq \int^1_0 \Bigg[ - \frac{t_2 - t_1}{\beta} \left(|\p f|_G - \frac{\beta |\dot{\gamma}|_G}{t_2 - t_1}\right)^2 + \frac{\beta |\dot{\gamma}|^2_G}{t_2 - t_1} \\
	&\hspace{3em} + C_1(t_2 - t_1) + C_2 \frac{t_2 - t_1}{t} \Bigg]ds \\
	&\leq \int^1_0 \Bigg[\frac{C_3}{t_2 - t_1} + C_1(t_2 - t_1) + C_2 \frac{t_2 - t_1}{t} \Bigg]ds \\
	&= \frac{C_3}{t_2 - t_1} + C_1 (t_2 - t_1) + C_2 \ln \frac{t_2}{t_1} .
\end{aligned}
\end{equation}

\end{proof}

\bigskip

\section{Convergence of the parabolic flow}
\label{pmach-convergence}
\setcounter{equation}{0}
\medskip

With the Harnack inequality, we can show the convergence of $\tilde u$ following Cao~\cite{Cao85}.

Define $\varphi = \frac{\p u}{\p t}$. Then we have
\begin{equation}
	\frac{\p\varphi}{\p t} = \sum_{i,j} G^{i\bar j} \bpartial_j\p_i \varphi.
\end{equation}

Let $m$ be a positive integer and define
\begin{equation}
\begin{aligned}
	\xi_m (x,t) &= \sup_{y\in M} \varphi (y, m - 1) - \varphi (x , m - 1 + t) \geq 0,\\
	\eta_m(x,t)&= \varphi(x, m - 1 + t) - \inf_{y\in M} \varphi(y, m - 1) \geq 0.
\end{aligned}
\end{equation}

These functions satisfy the heat equations
\begin{equation}
\begin{aligned}
	\frac{\p\xi_m}{\p t} &= \sum_{i,j} G^{i\bar j} (m - 1 + t) \bpartial_j\p_i \xi_m ,\\
	\frac{\p\eta_m}{\p t}&= \sum_{i,j} G^{i\bar j} (m - 1 + t) \bpartial_j\p_i \eta_m .
\end{aligned}
\end{equation}

If $\varphi(x , m - 1)$ is not a constant function, then $\xi_m > 0$ for some $x\in M$ at time $t = 0$. By the maximum principle, $\xi_m$ has to be positive for any $x \in M$ when $t > 0$. Similarly, $\eta_m$ is also positive everywhere in $M \times (0,\infty)$. Hence, we can apply Lemma~\ref{pmach-Harnack-lemma-3} with $t_1 = \frac{1}{2}$ and $t_2 = 1$,
\begin{equation}
\begin{aligned}
	&\sup_{M} \varphi(x , m - 1) - \inf_{M} \varphi\left(x , m - \frac{1}{2}\right) \leq C \left(\sup_{M} \varphi (x , m - 1) - \sup_{M} \varphi(x , m) \right), \\
	&\sup_{M} \varphi\left(x , m - \frac{1}{2}\right) - \inf_{M} \varphi(x , m - 1) \leq C \left(\inf_{M} \varphi(x , m) - \inf_{M} \varphi (x , m - 1)\right) .
\end{aligned}
\end{equation}
Define the oscillation $\theta(t) = \sup_{x\in M} \varphi(x,t) - \inf_{x\in M} \varphi(x,t)$. Adding the above inequalities gives us
\begin{equation}
	\theta(m - 1) + \theta\left(m - \frac{1}{2}\right) \leq C\left(\theta(m - 1) - \theta(m) \right) .
\end{equation}
So
\begin{equation}
	\theta(m) \leq \frac{C - 1}{C} \theta(m - 1).
\end{equation}
By induction, we have
\begin{equation}
	\theta (t) \leq C e^{- c_0 t} ,
\end{equation}
where $c_0 = - \ln \frac{C - 1}{C}$. On the other hand, if $\varphi(x , m - 1)$ is constant, this inequality still holds true since $\varphi(x,t)$ is then a constant function.


Since the normalized solution is defined by 
\begin{equation}
	\tilde u = u - \frac{\int_M u \omega^n}{\int \omega^n} ,
\end{equation}
it is obvious that
\begin{equation}
	\int_M\frac{\p \tilde u}{\p t} \omega^n = 0.
\end{equation}
Fixing $(x,t)$ in $M\times[0,\infty)$, there is a point $y\in M$ such that 
\begin{equation}
	\frac{\p\tilde u}{\p t} (y, t) = 0.
\end{equation}
Hence,
\begin{equation}
\begin{aligned}
	\left|\frac{\p\tilde u}{\p t}(x , t)\right| & = \left|\frac{\p\tilde u}{\p t}(x , t) - \frac{\p\tilde u}{\p t}(y , t)\right| \\
	& = \left|\frac{\p u}{\p t}(x , t) - \frac{\p u}{\p t}(y , t)\right| \\
	& \leq C e^{- c_0 t} .
\end{aligned}
\end{equation}
Consider the quantity $Q = \tilde u + \frac{C}{c_0} e^{- c_0 t}$,
\begin{equation}
	\frac{\p Q}{\p t} \leq 0.
\end{equation}
Since $Q$ is bounded and pointwise monotonically decreasing, it tends to a limit as $t\rightarrow \infty$, say, $\tilde u_\infty$. But
\begin{equation}
	\lim_{t\rightarrow \infty} \tilde u = \lim_{t\rightarrow \infty} Q = \tilde u_\infty.
\end{equation}

We show that the convergence of $\tilde u$ to $\tilde u_\infty$ is actually $C^\infty$ by contradiction. Suppose that the convergence is not $C^\infty$, then there exists a time sequence $t_i \rightarrow \infty$ such that for some $\epsilon > 0$ and some integer $k$,
\begin{equation}
||\tilde u(x , t_i) - \tilde u_\infty ||_{C^k} > \epsilon, \quad \forall i.
\end{equation}
Since $\tilde u$ is bounded in $C^\infty$, there exists a subsequence $t_{i_j} \rightarrow \infty$ such that $\tilde u (x , t_{i_j}) \rightarrow \tilde U_\infty$ as $j\rightarrow \infty$ for some smooth function $\tilde U_\infty$. Definitely, $\tilde u_\infty \neq \tilde U_\infty$. But it contradicts the fact that $\tilde u \rightarrow \tilde u_\infty$ pointwise.

Note that $\tilde u$ solves the parabolic flow
\begin{equation}
	\frac{\p \tilde u}{\p t} = \ln  \frac{\chi_{\tilde u}^n }{\chi_{\tilde u}^{n - \a} \wedge \omega^\a} - \ln \psi  - \frac{\int_M u \omega^n}{\int \omega^n} .
\end{equation}
Letting $t\rightarrow \infty$,  $\tilde u_\infty$ solves the equation
\begin{equation}
	\ln  \frac{\chi_{\tilde u}^n }{\chi_{\tilde u}^{n - \a} \wedge \omega^\a} = \ln \psi  + b ,
\end{equation}
where
\begin{equation}
	b = \frac{\int_M \left(\ln  \frac{\chi_{\tilde u}^n }{\chi_{\tilde u}^{n - \a} \wedge \omega^\a} - \ln \psi\right) \omega^n}{\int_M \omega^n} = \lim_{t\rightarrow \infty} \frac{\int_M \left(\ln  \frac{\chi_{u}^n }{\chi_{ u}^{n - \a} \wedge \omega^\a} - \ln \psi\right) \omega^n}{\int_M \omega^n} .
\end{equation}
This completes the proof the convergence.

\bigskip

\section{Revisit to method of continuity}
\label{pmach-revisit}
\setcounter{equation}{0}
\medskip

In this section, we want to avoid using the specific knowledge of $J$-functionals, Gauduchon's theorem~\cite{Ga77} and Buchdahl's result~\cite{Buchdahl99}. So we apply the parabolic result, Theorem~\ref{pmach-int-theorem-convergence-Hermitian}, to carrying out method of continuity.

Let us recall the {\em a priori} esimates for elliptic complex Monge-Amp\`ere type equations in \cite{Sun2013e}.
\begin{theorem}.
\label{pmach-theorem-elliptic-estimate}
Let $(M^n,\omega)$ be a closed Hermitian manifold of complex dimension $n$ and $u$ be a smooth solution of the equaion~\eqref{pmach-elliptic-equation}. Suppose that $\chi \in \mathscr{C}_\a (\psi)$. Then there are uniform $C^\infty$ a priori estimates of $u$.
\end{theorem}

First, since $\chi \in \mathscr{C}_\a (\psi)$, there must be a function $\ul u$ satisfying 
\begin{equation}
	\chi_{\ul u} = \chi + \frac{\sqrt{-1}}{2} \p\bpartial\ul u > 0
\end{equation}
and 
\begin{equation}
	n \chi^{n - 1}_{\ul u} > (n - \a) \psi \chi^{n - \a - 1}_{\ul u} \wedge \omega^\a .
\end{equation}
As in \cite{Sun2013e}, without loss of generality, we can assume $\ul u$ is smooth.

Define $\ul \varphi$ by
\begin{equation}
	\chi^n_{\ul u} = \ul\varphi \chi^{n - \a}_{\ul u} \wedge \omega^\a .
\end{equation}
It follows from the monotonicity of $\frac{S_n}{S_{n - \a}}$ that
\begin{equation}
	n \chi^{n - 1}_{\ul u} > (n - \a) \ul\varphi \chi^{n - \a - 1}_{\ul u} \wedge \omega^\a .
\end{equation}
Consequently, 
\begin{equation}
	n \chi^{n - 1}_{\ul u} > (n - \a) (\max\{\psi, \ul \varphi\} + 2 \delta)\chi^{n - \a - 1}_{\ul u} \wedge \omega^\a .
\end{equation}
for sufficiently small $\delta > 0$. By approximation, we can find a smooth function $\psi_0$ satisfying 
\begin{equation}
	\max\{\psi , \ul\varphi\} \leq \psi_0 \leq \max\{\psi , \ul\varphi\} + \delta.
\end{equation}
 So we have $\chi_{\ul u} \in \mathscr{C}_\a (\psi_0)$ and
\begin{equation}
\frac{\chi^n_{\ul u}}{\chi^{n - \a}_{\ul u} \wedge \omega^\a} \leq \psi_0.
\end{equation}
By Theorem~\ref{pmach-int-theorem-convergence-Hermitian} and maximum principle, there exists an admissible solution $u_0$ of
\begin{equation}
    \left(\chi + \frac{\sqrt{-1}}{2}\p\bpartial u \right)^n = \psi_0 e^{b_0}\left(\chi + \frac{\sqrt{-1}}{2}\p\bpartial u\right)^{n - \a} \wedge \omega^\a , 
\end{equation}
for some $b_0 \leq 0$.

Second, we start the method of continuity from $\chi_{u_0}$ and consider the family of equations
\begin{equation}
\label{pmach-equation-continuity-method-1}
    \left(\chi + \frac{\sqrt{-1}}{2}\p\bpartial u_s\right)^n = \psi^s \psi_0^{1-s} e^{b_s} \left(\chi + \frac{\sqrt{-1}}{2}\p\bpartial u_s\right)^{n - \a} \wedge \omega^\a , \quad\text{ for } s\in [0,1].
\end{equation}
Note that $b_0$ has been found out in the first stage.

Integrating equation \eqref{pmach-equation-continuity-method-1},
\begin{equation}
	\int_M \chi^{n} = \int_M \psi^s \psi_0^{1-s} e^{b_s} \left(\chi + \frac{\sqrt{-1}}{2}\p\bpartial u_s\right)^{n - \a} \wedge \omega^\a \geq c e^{b_s} \int_M \chi^{n - \a} \wedge \omega^\a ,
\end{equation}
which implies
\begin{equation}
	b_s \leq 0.
\end{equation}
So we have
\begin{equation}
	 \psi^s \psi_0^{1-s} e^{b_s} \leq  \psi^s \psi_0^{1-s} \leq \psi_0.
\end{equation}
Therefore, there are uniform $C^\infty $ estimates of $u_s$, as a consequnce of Theorem~\ref{pmach-theorem-elliptic-estimate}.

We consider the set
\begin{equation}
\label{pmach-continuity-method-definition-S}
    \mathcal{S} := \{s'\in[0,1]\;|\; \exists \; u_s \in C^{2,\a}(M) \text{ and } b_s \text{ solving } \eqref{pmach-equation-continuity-method-1} \text{ for } s\in[0,s']\}.
\end{equation}
Since we obtained $b_0$ in the first stage, $0 \in \mathcal{S}$ and hence $\mathcal{S}$ is not empty. It suffices to show that $\mathcal{S}$ is both open and closed in $[0,1]$.

From equation \eqref{pmach-equation-continuity-method-1}, we have 
\(
 \psi^s \psi_0^{1 - s} e^{b_s} \geq \psi_0 e^{b_0}
\)
when $u_s$ achieves its minimum. 
So
\begin{equation}
	0 \geq b_s \geq \inf_M \left(\ln \psi_0 - \ln \psi\right) + b_0.
\end{equation}
The closedness of $\mathcal{S}$ follows from the uniform bound for $b_s$ and uniform $C^\infty$ estimates of $u_s$.

Now we show that $\mathcal{S}$ is open. Assuming that $\hat s \in \mathcal{S}$, we need to show that there exists small $\epsilon > 0$ such that $s\in\mathcal{S}$ for any $s\in[\hat s , \hat s + \epsilon)$.
\begin{equation}
    \left(\chi + \frac{\sqrt{-1}}{2}\p\bpartial u_{\hat s}\right)^n = \psi^{\hat s} \psi_0^{1-{\hat s}} e^{b_{\hat s}} \left(\chi + \frac{\sqrt{-1}}{2}\p\bpartial u_{\hat s}\right)^{n - \a} \wedge \omega^\a .
\end{equation}
When $\psi\equiv\psi_0$, the openess is obvious. If $\psi\not\equiv\psi_0$, we choose a positive constant $\kappa > 1$ satisfying
\begin{equation}
	\kappa \psi_0 \leq \psi_0 + \delta.
\end{equation}
Thus, we have
\begin{equation}
	\kappa \psi^s \psi^{1 - s}_0 \leq \psi_0 + \delta 
\end{equation}
and hence
\begin{equation}
\label{pmach-method-of-continuity-flow-condition-1}
	n \chi^{n - 1}_{\ul u} > (n - \a) \kappa \psi^s \psi^{1 - s}_0 \chi^{n - \a - 1}_{\ul u} \wedge \omega^\a .
\end{equation}
Also, 
\begin{equation}
\label{pmach-method-of-continuity-flow-condition-2}
\begin{aligned}
	\kappa \psi^s \psi^{1 - s}_0 &\geq \kappa e^{- b_{\hat s}} \left(\frac{\psi}{\psi_0}\right)^{s - \hat s} \psi^{\hat s} \psi^{1 - \hat s}_0 e^{b_{\hat s}} \\
	&\geq \kappa \left(\frac{\psi}{\psi_0}\right)^{s - \hat s} \psi^{\hat s} \psi^{1 - \hat s}_0 e^{b_{\hat s}} \\
	&\geq \kappa \left(\inf_M\frac{\psi}{\psi_0}\right)^{s - \hat s} \psi^{\hat s} \psi^{1 - \hat s}_0 e^{b_{\hat s}} \\
	&\geq \psi^{\hat s} \psi^{1 - \hat s}_0 e^{b_{\hat s}} \\
	&= \frac{\chi^n_{u_{\hat s}}}{ \chi^{n - \a}_{u_{\hat s}} \wedge \omega^\a }
\end{aligned}
\end{equation}
whenever
\begin{equation}
	s - \hat s \leq - \frac{\ln \kappa}{\ln \inf_M \frac{\psi}{\psi_0}}.
\end{equation}
Let  $\epsilon = - \frac{\ln \kappa}{ \inf_M (\ln\psi - \ln\psi_0)}$ and consider the parabolic equation
\begin{equation}
\frac{\p u}{\p t} = \ln  \frac{\chi_u^n }{\chi_u^{n - \a} \wedge \omega^\a} - \ln \left(\kappa \psi^s \psi^{1 - s}_0\right) ,
\end{equation}
with $u(x,t) = u_{\hat s}(x)$. Therefore, from \eqref{pmach-method-of-continuity-flow-condition-1}, \eqref{pmach-method-of-continuity-flow-condition-2} and Theorem~\ref{pmach-int-theorem-convergence-Hermitian}, we have a pair $(u_s , b_s)$ solving equation~\eqref{pmach-equation-continuity-method-1} for $s\in[\hat s , \hat s + \epsilon)$, and hence $\mathcal{S}$ is open.

\end{document}